\documentclass[a4paper,11pt]{scrartcl}
\usepackage[utf8]{inputenc}
\usepackage{amssymb,amsfonts,amsmath,amsthm,graphicx}
\usepackage{tikz}
\usetikzlibrary{matrix}
\usetikzlibrary{arrows,decorations.pathmorphing}

\newtheorem{theorem}{Theorem}
\newtheorem{example}{Example}
\newtheorem{lemma}[theorem]{Lemma}
\newtheorem{proposition}[theorem]{Proposition}

\newtheorem{remark}[theorem]{Remark}

\renewcommand{\int}{\mathsf{int}}

\newcommand{\sep}{\mathsf{sep}}
\newcommand{\fsep}{\mathsf{fsep}}

\title{Choosability with Separation of Cycles and Outerplanar Graphs}
\author{Jean-Christophe Godin $^a$ \and Olivier Togni $^b$}

\begin{document}

\maketitle
\begin{center}
$^a$ Institut de Math\'ematiques de Toulon, Universit\'e de Toulon, France\\
\texttt{godinjeanchri@yahoo.fr}
\medskip

$^b$  Laboratoire LIB, Université Bourgogne Franche-Comté, France\\
\texttt{olivier.togni@u-bourgogne.fr}
\medskip
\end{center}

\begin{abstract}
We consider the following list coloring with separation problem of graphs: Given a graph $G$ and integers $a,b$, find the largest integer $c$ such that for any list assignment $L$ of $G$ with $|L(v)|\le a$ for any vertex $v$ and $|L(u)\cap L(v)|\le c$ for any edge $uv$ of $G$, there exists an assignment $\varphi$ of sets of integers to the vertices of $G$ such that $\varphi(u)\subset L(u)$ and $|\varphi(v)|=b$ for any vertex $v$ and $\varphi(u)\cap \varphi(v)=\emptyset$ for any edge $uv$. Such a value of $c$ is called the separation number of $(G,a,b)$. We also study the variant called the free-separation number which is defined analogously but assuming that one arbitrary vertex is precolored. We determine the separation number and free-separation number of the cycle and derive from them the free-separation number of a cactus. We also present a lower bound for the separation and free-separation numbers of outerplanar graphs of girth $g\ge 5$.
\end{abstract}

\section{Introduction}

Let $a,b,c$ and $k$ be integers and let $G$ be a graph. A $k$-list assignment $L$ of $G$ is a function which associates to each vertex a set of at most $k$ integers. The list assignment $L$ is {\em $c$-separating} if for any $uv\in E(G)$, $|L(u)\cap L(v)|\le c$. The graph $G$ is {\em $(a,b,c)$-choosable} if for any $c$-separating $a$-list assignment $L$, there exists an $(L,b)$-coloring of $G$, i.e. a coloring function $\varphi$ on the vertices of $G$ that assigns to each vertex $v$ a subset of $b$ elements from $L(v)$ in such a way that $\varphi(u)\cap \varphi(v)=\emptyset$ for any $uv\in E(G)$.

This type of restricted list coloring problem, called choosability with separation, has been introduced by Kratochvíl, Tuza and Voigt~\cite{KTV98a}. Notice that Kratochvíl et al.~\cite{KTV98a, KTV98} defined $(a,b,c)$-choosability a bit differently, requiring for a $c$-separating $a$-list assignment $L$ that the lists of two adjacent vertices $u$ and $v$ satisfy $|L(u)\cap L(v)|\le a-c$. Among the first results on the topic, a complexity dichotomy was presented~\cite{KTV98a} and general properties given~\cite{KTV98}.  Since then, a number of papers has considered choosability with separation of planar graphs, mainly for the case $b=1$~\cite{BCD+, CLW18, CYR+, CLS16, CFWW, KL15, Skr01}. A still open question for this class of graph is whether all planar graphs are $(4,1,2)$-choosable or not. Other recent papers concern balanced complete multipartite graphs and $k$-uniform hypergraphs (for the case $b=1$)~\cite{FKK}; bipartite graphs (for the case $b=c=1$)~\cite{EKT19} and a study with an extended separation condition~\cite{KMS}.

In this paper, we concentrate on choosability and free-choosability with separation of cycles and outerplanar graphs in a little different point of view: as a $(a,b,c)$-choosable graph is also $(a,b,c')$-choosable for any $c'<c$, our aim is to determine, for given $a,b$, $a\ge b$, the largest $c$ such that $G$ is $(a,b,c)$-choosable. We find convenient to define the parameter $\sep(G,a,b)$ that we call the {\em  (list) separation number} of $G$ as
\[\sep(G,a,b)=\max\{c, G \text{ is } (a,b,c)\text{-choosable}\}.\]

The notion of free choosability~\cite{AGT16}, that consists in considering list assignments on graphs with a precolored vertex, easily extends to choosability with separation: a graph $G$ is {\em $(a,b,c)$-free-choosable} if for any $c$-separating $a$-list assignment $L$, any $v\in V(G)$ and any $C\subset L(v)$ with $|C|=b$, there exists an $(L,b)$-coloring $\varphi$ such that $\varphi(v)=C$. Alternatively, we can view free-choosability as classical choosability but with a list of cardinality $b$ on one arbitrary vertex. Analogously with the separation number, we define the {\em free-separation number} $\fsep(G,a,b)$ of a graph $G$ as 
\[\fsep(G,a,b)=\max\{c, G \text{ is } (a,b,c)\text{-free-choosable}\}.\]


Clearly, for any graph $G$ and any integers $a$ and $b$, we have $\sep(G,a,b)\ge \fsep(G,a,b)$. Moreover, since for any $a\ge b\ge 1$, every graph $G$ is $(a,b,0)$-free-choosable, we have $$0\le \fsep(G,a,b)\le \sep(G,a,b) \le a$$ and thus both parameters are well defined for any graph. As a first example (and as we will prove in Proposition~\ref{prop:sepCact}), the graph $G$ depicted in Figure~\ref{fig:2C4} is not $(2,1,1)$-choosable, thus implying that $\sep(G,2,1)=\fsep(G,2,1)=0 < \sep(C_4,2,1)=1$.
%

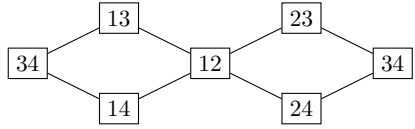
\begin{figure}[t]
\begin{center}
\begin{tikzpicture}[scale=1.2]
 \node at (0,0)[draw=black,scale=0.8](a1){$34$};
 \node at (1,0.5)[draw=black,scale=0.8](a2){$13$};
 \node at (1,-0.5)[draw=black,scale=0.8](a3){$14$};
 \node at (2,0)[draw=black,scale=0.8](x){$12$};
 \node at (3,0.5)[draw=black,scale=0.8](b2){$23$};
 \node at (3,-0.5)[draw=black,scale=0.8](b3){$24$};
 \node at (4,0)[draw=black,scale=0.8](b1){$34$};
 \draw (a1) -- (a2) -- (x) -- (a3) -- (a1);
 \draw (b1) -- (b2) -- (x) -- (b3) -- (b1);
\end{tikzpicture}
\end{center}
\caption{\label{fig:2C4}A cactus with a $1$-separating $2$-list assignment $L$ for which no $(L,1)$-coloring exists.}
\end{figure}

In this paper, we determine the separation number of the cycle in Section~\ref{sec:sepCn} (Theorem~\ref{th:Cn}) and free-separation number (Theorem~\ref{thm:fCn} and Proposition~\ref{prop:fC3}) of the cycle in Section~\ref{sec:fsepCn}. Contrary to the separation number, we show that the free-separation number of a cycle does not depend on the parity of its length and that $C_3$ is a special case. We then use these results to determine bounds and exact values for the same invariants on outerplanar graphs of girth at least 5 and tighter bounds for the subclass of cactuses in Section~\ref{sec:fsepCact}. Some possible directions for further works are given in Section 5.

Our proofs are all constructive and the proofs of upper bounds on $\sep$ and $\fsep$ rely on finding counter-example list assignments. These examples are constructed in a greedy way, maximizing for each list, the intersections with lists of other adjacent vertices (while satisfying the $c$-separating condition). Our proofs for $\fsep(C_n,a,b)$ use special types of list assignments of the path (Lemmas~\ref{lem:Pn+1} and~\ref{lem:Pn+1onlyif}) that may be of interest for obtaining other choosability results.

\section{Separation number of the cycle}
\label{sec:sepCn}
Using a similar argument than the one used by Kratochv\'il et al.~\cite{KTV98} (in the more general setting of graphs with bounded outdegree orientation) we have the following:

\begin{proposition}\label{prop:Cn}
For any $n\ge 3$ and $a\ge b$, we have $\sep(C_n,a,b)\ge a-b$.
\end{proposition}

\begin{proof} Let $L$ be a $k$-separating $(b+k)$-list assignment. Orient $C_n$ clockwise, with $x^-$ and $x^+$ being the predecessor and successor of vertex $x$, respectively. Since for any $x\in V(C_n)$, $|L(x)\cap L(x^+)|\le k$, we have $|L(x)\setminus L(x^+)|\ge b+k-k= b$. Hence it is possible to assign a set $\varphi(x)$ of $b$ colors from $L(x)\setminus L(x^+)$ to each vertex $x$. 

Since any $(a,b,c)$-choosable graph is also $(a',b,c)$-choosable for any $a'\ge a$, we obtain that $C_n$ is $(a,b,c)$-choosable when $c\le a-b$.
\end{proof}

Since $a\ge b$, we can rewrite the above inequality as $\sep(C_n,b+k,b)\ge k$ for any $k\ge 0$.

As the next result shows, the above result is tight provided $k<b$.
\begin{proposition}\label{prop:CnLower}
For any $n\ge 3, b\ge 1$ and $k< b$, we have $\sep(C_{n}, b+k,b)\le k$.
\end{proposition}

\begin{proof}
We provide a  $(k+1)$-separating $(b+k)$-list assignment $L$ for which no $(L,b)$-coloring of $C_{n}$ exists. Let $X$ be a set of $n(b-1)+1$ colors and let $C, D_i, F_i$, $i\in \{0,\ldots, n-1\}$ be a partition of $X$ with $|C|=1, |D_i|=k, |F_i|=b-k-1$. Let $C_{n}=(x_0,\ldots, x_{n-1})$ and for any $i, 0\le i\le n-1$, let 

$$L(x_i)=C\cup D_i\cup D_{i+1}\cup F_i,$$
with indices taken modulo $n$.

Now, observe that, by the construction of the list assignment $L$, the color of $C$ is present in the color-list of every vertex and that every color of any set $D_i$ is present in the lists of two consecutive vertices. Therefore, the color of $C$ can be assigned to at most $\lfloor n/2\rfloor$ vertices while every color of every set $D_i$, $i=0,\ldots,n-1$ can be given to at most one vertex. Hence, the total number of colors that can be given to vertices of $C_{n}$ is $\lfloor n/2\rfloor+nk + n(b-k-1)= \lfloor n/2\rfloor+ n(b-1)<nb$. Then, since the $n$ vertices of $C_{n}$ require $nb$ colors in total, no $(L,b)$-coloring of $C_{n}$ exists.
\end{proof}

Reusing the method of the proof of Proposition~\ref{prop:Cn} with a little more involved argument, we are able to prove:
\begin{proposition}\label{prop:Cn+2k}
For any $a,b,c$, $n\ge 3$ and $k\ge 1$, the following implication is true for the cycle $C_n$:
 $$C_n \ (a,b,c)\text{-choosable} \Rightarrow C_n \ (a+2k,b+k,c+k)\text{-choosable}.$$
\end{proposition}

\begin{proof} Suppose $a\ge c$ and $C_n$ is $(a,b,c)$-choosable and let $L$ be a $(c+k)$-separating $(a+2k)$-list assignment of $C_n$. Orient $C_n$ clockwise, with $x^-$ and $x^+$ being the predecessor and successor of vertex $x$, respectively. Since for any $x\in V(C_n)$, $|L(x)\cap L(x^+)|\le c+k$, we have $|L(x)\setminus L(x^+)|\ge a+2k-(c+k)= a-c+k\ge k$. Hence it is possible to assign a set $\varphi(x)$ of $k$ colors from $L(x)\setminus L(x^+)$ to each vertex $x$. 

Now we still have to assign $b$ more colors to each vertex to complete the coloring. For this, we construct a new list assignment $L'$ by removing from $L(x)$ the colors already assigned to $x$ and also a maximum number of colors from $L(x)\cap L(x^+)$, including those assigned to $x^+$ that are in $L(x)$, if any, in order $L'$ to be a $c$-separating $a'$-list assignment, with $a'\ge a$ (observe that, by construction, $L(x)\cap \varphi(x^-)=\emptyset$):
For any $x\in V(C_n)$, let $I^+(x)=L(x)\cap L(x^+)$ and let $S(x)$ be any subset of $I^+(x)$ of size $\min\{k, |I^+(x)|\}$ that contains $L(x)\cap \varphi(x^+)$. (Note that $I^+(x)$ and $S(x)$ may be empty.) Then, define a new list assignment $L'$ on $C_n$ by: $\forall x\in V(C_n)$,

$$ L'(x)=L(x)\setminus (\varphi(x) \cup S(x)).$$

We then have $|L'(x)|= a+2k -k -\min\{k, |I^+(x)|\} \ge a$ and $|L'(x)\cap L'(x^+)| \le c+k -\min\{k, |I^+(x)|\} \le c$, hence $L'$ is a $c$-separating $a'$-list assignment with $a'\ge a$. Therefore, since by hypothesis $G$ is $(a,b,c)$-choosable, there exists an $(L',b)$-coloring of $G$, which together with the coloring $f$, produces an $(L,b+k)$-coloring of $G$, proving that $G$ is $(a+2k,b+k,c+k)$-choosable.
\end{proof}

For cycles of odd length, combining known choosability results on the cycle with Proposition~\ref{prop:Cn+2k} allows to determine the separation number in the remaining cases:
\begin{proposition}\label{prop:C2p+1}
For any integers $b, p$ and $\alpha$ with $0\le \alpha \le \frac{b}{p}$ and $p\ge 1$, we have 
\[ \sep(C_{2p+1},2b+\alpha,b)=b+(p+1)\alpha. \]
\end{proposition}

\begin{proof}
It is known that cycles of length $2p+1$ are $(a,b,a)$-choosable if and only if $a/b\ge 2+1/p$ (see~\cite{AGT16}). Hence, for $\alpha\ge 0$, $C_{2p+1}$ is $((2p+1)\alpha,p\alpha,(2p+1)\alpha)$-choosable. Then by Proposition~\ref{prop:Cn+2k}, $C_{2p+1}$ is also $((2p+1)\alpha+2k,p\alpha+k,(2p+1)\alpha+k)$-choosable for any $k\ge 0$. Setting $k=b-p\alpha$, we obtain that $C_{2p+1}$ is $(2b+\alpha, b,b+(p+1)\alpha)$-choosable.

Note that if $p$ divides $b$ and if $\alpha=b/p$, then $b+(p+1)\alpha = 2b+\alpha$. Hence in this case, $\sep(C_{2p+1},2b+\alpha,b)=2b+\alpha$. Otherwise, to prove that $C_{2p+1}$ is not $(2b+\alpha, b,b+(p+1)\alpha+1)$-choosable for $\alpha \le \frac{b-1}{p}$, we provide a $(b+(p+1)\alpha+1)$-separating $(2b+\alpha)$-list assignment $L$ of $C_{2p+1}$ for which no $(L,b)$-coloring exists. Let $C$ be a set of $(2p+1)\alpha + 2$ colors and $D_i, i=0,\ldots, 2p$ be $2p+1$ pairwise disjoint sets of $b-p\alpha-1$ colors (also disjoint from $C$). Let  $C_{2p+1}=(x_0,\ldots, x_{2p})$ and set $$L(x_i)= C \cup D_i\cup D_{i+1},$$ with $0\le i\le 2p$ and indices taken modulo $2p+1$. 

It can be checked that the lists of any two consecutive vertices have $(2p+1)\alpha +2+b-p\alpha-1=b+(p+1)\alpha+1$ elements in common. Assume now that there exists an $(L,b)$-coloring of $C_{2p+1}$. Observe that, by the construction of the list assignment $L$, every color from $C$ is present in the color-list of every vertex and that every color of any set $D_i$ is present in the lists of two consecutive vertices. Therefore, every color of $C$ can be assigned to at most $p$ vertices while every color of every set $D_i$, $i=0,\ldots,2p$ can be given to at most one vertex. Hence we have $p((2p+1)\alpha+2) + (2p+1)(b-p\alpha-1)=b(2p+1)-1$ available colors in total but we have to assign $b(2p+1)$ colors to the vertices, a contradiction.
\end{proof}

We are now ready to determine the separation number of the cycle.
\begin{theorem}\label{th:Cn}
For any $p\ge 1$ and any $a,b$ such that $a\ge b\ge 1$, 
 \[\sep(C_{2p+2},a,b) = \left\{\begin{array}{ll}
a-b, & b\le a< 2b\\
a, & a\ge 2b.\\ 
                            \end{array}
\right.\]

\[ \sep(C_{2p+1},a,b) = \left\{\begin{array}{ll}
a-b, & b\le a< 2b\\
b+(p+1)(a-2b), & 2b \le a \le 2b + \frac{b}{p}\\
a, &  a \ge 2b + \frac{b}{p}. 
                            \end{array}
\right.\]
\end{theorem}

\begin{proof}
 For even-length cycles, the result is obtained by combining Propositions~\ref{prop:Cn} and~\ref{prop:CnLower} and noting that $C_{2p}$ is $(2b,b,2b)$-choosable.
 
 For odd-length cycles, the combination of Propositions~\ref{prop:Cn}, ~\ref{prop:CnLower} and~\ref{prop:C2p+1} and the known fact that $C_{2p+1}$ is $(a,b,a)$-choosable for any $a,b$ such that $a/b\ge 2 + 1/p$ leads to the result.
\end{proof}

\begin{example}
 From the above theorem, we know that the cycle $C_3$ is $(5,2,4)$-choosable but not $(5,2,5)$-choosable (this was already known) and also $(7,3,5)$-choosable but not $(7,3,6)$-choosable (while it was only known before that $C_3$ is not $(7,3,7)$-choosable). In contrast, $C_5$ is $(7,3,6)$-choosable and $(9,4,7)$-choosable but not $(7,3,7)$-choosable and not $(9,4,8)$-choosable. 
\end{example}

\section{Free separation number of the cycle}
\label{sec:fsepCn}

In order to determine the free-separation number of the cycle, we first set some notation and preliminary results. 

The following Hall-type condition that we call the {\em amplitude condition} is necessary for a graph $G$ to be $(L,b)$-colorable: 
$$\forall H\subset G, \sum_{k\in C} \alpha(H,L,k) \ge b|V(H)|,$$ where $C=\bigcup_{v\in V(H)}L(v)$ and $\alpha(H,L,k)$ is the independence number of the subgraph of $H$ induced by the vertices containing $k$ in their color list. Notice that $H$ can be restricted to be a connected induced subgraph of $G$. As shown by Cropper et al.~\cite{CGHHJ} (in the more general context of weighted list coloring), this condition is also sufficient when the graph is a complete graph or a path (or some other specific graphs).

For a list assignment $L$ on a graph $G$ of order $n$ with vertex set $V(G)=\{x_1,x_2, \ldots, x_n\}$, we let $L_i=L(x_i)$ and for $1\le i<j\le n$, we write $\Sigma_{i,j}(L)=\sum_{k\in C}\alpha(H,L,k)$, where $H$ is the subgraph of $G$ induced by vertices $x_i, x_{i+1},\ldots, x_j$. We also simplify $\Sigma_{1,n}(L)$ to $\Sigma(L)$.

From now on, a cycle of order $n$ will have its vertices denoted by $x_1,x_2,\ldots, x_n$ following some order on the cycle and the vertices of a path of order $n$ will be also denoted by $x_1,x_2,\ldots, x_n$ following the path from one end-vertex to the other.
 We will use the following relation to show that the cycle is $(L,b)$-colorable for some lists $L$. 

\begin{remark}\label{rem}
 A list assignment $L$ on the cycle $C_n$ with $|L_1|=b$ and $|L_i|=a$ for any $i, 2\le i\le n$ can be transformed into a list assignment $L'$ on the path $P_{n+1}$  by setting $L'_{n+1}=L_1$ and $L'_i=L_{i}$ for $1 \le i\le n$ (i.e., $P_{n+1}$ has been obtained by ``cutting'' the cycle on the vertex $x_1$). Clearly if $P_{n+1}$ is $(L',b)$-colorable, then $C_n$ is $(L,b)$-colorable.
\end{remark}

As induced subgraphs of paths are sub-paths, the amplitude condition for a path $P_n$ can be rewritten as : \[ \forall i,j, \ 1\le i\le j\le n,\ \Sigma_{i,j}(L) \ge b(j-i+1). \]

For integers $n,a$ and $b$ with $b\le a$, let 
\[c(n,a,b) = \left\{\begin{array}{lcccl}
\frac{n-1}{n}(a-b), & \mbox{ if } & b&\le &a< \frac{2n-1}{n-1}b\\
\frac{n-1}{n-2}(a-b) - \frac{2}{n-2}b, & \mbox{ if } & \frac{2n-1}{n-1}b &\le &a < 2\frac{n+1}{n}b\\
a, &  \mbox{ if } &2\frac{n+1}{n}b &\le & a.
                            \end{array}\right. \]

We now prove two lemmas about conditions for a path with precolored endvertices to be list colorable.
\begin{lemma}\label{lem:Pn+1}
 Let $n, a,b,c$ be integers, $n\ge 3$, $b\le a < 2\frac{n+1}{n}b$ and $c=\lfloor c(n,a,b)\rfloor$. For any $c$-separating $a$-list assignment $L$ of $P_{n+1}$ such that $|L_1|=|L_{n+1}|=b$, there exists an $(L,b)$-coloring of $P_{n+1}$.
\end{lemma}

\begin{proof}
It is sufficient to verify that the amplitude condition is satisfied. First, we show that in both cases, we have $c\le a-b$. If $c\le \frac{n-1}{n}(a-b)$ then clearly $c\le a-b$. If $c\le \frac{n-1}{n-2}(a-b) - \frac{2}{n-2}b = \frac{n-2}{n-2}(a-b)+ \frac{a-b}{n-2} -\frac{2}{n-2}b= a-b +\frac{a-3b}{n-2}$, then again, $c\le a-b$ since $a < 2 \frac{n+1}{n}b \le 3b$ as soon as $n\ge 3$.

Now, since $L$ is $c$-separating then if $j< n$ we have 
$$|L_{j+1}\setminus L_j|\ge a-c$$
and thus, for any $1\le i\le j$, 
\begin{equation}\label{eq:a-c}
 \Sigma_{i,j+1}(L)\ge \Sigma_{i,j}(L) + a-c
\end{equation}

Moreover, since $|L_{n+1}|=b$, we also have 
\begin{equation}\label{eq:b-c}
 \Sigma_{i,n+1}(L)\ge \Sigma_{i,n}(L) + \max(b-c,0)
\end{equation}

Therefore, if  $1 <  i\le j\le n$, then $\Sigma_{i,j}(L) \ge (j-i+1) (a-c)$. Hence the amplitude condition is satisfied in this case if $(j-i+1) (a-c)\ge (j-i+1)b$, i.e., if $a-c\ge b$, which is true since we have shown above that $c\le a-b$.

If $i=1$ or $j=n+1$, we consider two cases depending on $a$, $b$ and $c$ (note that the ratio $\frac{2n-1}{n-1}$ has been chosen in such a way that $c<b$ in Case 1 and $c\ge b$ in Case 2).

 \begin{description}
  \item[Case 1.] $c=\lfloor \frac{n-1}{n}(a-b)\rfloor$ and $b\le a< \frac{2n-1}{n-1}b$. 
  
The above imply $c\le  \frac{n-1}{n}(a-b) < \frac{n-1}{n}\frac{2n-1}{n-1} b - \frac{n-1}{n}b = b$.
  Hence, if $i=1$ and $2\le j\le n$, then by Equation~\ref{eq:a-c}, $\Sigma_{1,j}(L) \geq b+ (a-c)j  \ge (j+1)b,$ since $a-c\ge b$. Otherwise, if $i>1$ and $j= n+1$, then $\Sigma_{i,n+1}(L) \geq a+ (a-c)(n-i)+b-c = b+(a-c)(n+1-i)$. Hence, since $a-c\ge b$, we have that $\Sigma_{i,n+1}(L) \ge (n+2-i)b$. Finally, since $c\le b$, by \ref{eq:a-c} and~\ref{eq:b-c}, we infer $\Sigma(L) \geq b+ (a-c)(n-1)+b-c  \ge (n+1)b$ if $c\le \frac{n-1}{n}(a-b)$.

  \item[Case 2.] $c= \lfloor\frac{n-1}{n-2}(a-b) - \frac{2}{n-2}b\rfloor$ and $a\ge \frac{2n-1}{n-1}b$.  
  
  In this case we have $c>\frac{n-1}{n-2}(a-b) - \frac{2}{n-2}b -1\ge \frac{n-1}{n-2}\frac{2n-1}{n-1}b - \frac{n+1}{n-2}b-1 = b-1$. Hence, if $i=1$ and $2\le j\le n$, then $\Sigma_{1,j}(L) \geq b+ a-b + (a-c)(j-2) \ge b+ (a-c)(j-1)  \ge jb,$ since $a-c\ge b$. Otherwise, if $i>1$ and $j= n+1$, then $\Sigma_{i,n+1}(L) \geq b+ a-b + (a-c)(n-i) \ge b+(a-c)(n+1-i) \ge (n+2-i)b$ since $a-c\ge b$. Finally, $\Sigma_{1,n+1}(L) \geq \Sigma_{2,n}(L)\ge a+ (a-c)(n-2)  \ge (n+1)b$ since $c\le \frac{n-1}{n-2}(a-b) -\frac{2}{n-2}b$.
 \end{description}
 
 Therefore, in both cases, the amplitude condition is satisfied hence $P_{n+1}$ is $(L,b)$-colorable.
\end{proof}

As shown by the next lemma, the above condition in Lemma~\ref{lem:Pn+1} is also necessary (provided that $n\ge 4$) and even with restricted list assignments in which the lists of the two endvertices are the same or are disjoint.

\begin{lemma}\label{lem:Pn+1onlyif}
Let $n, a,b,c$ be integers, $n\ge 4$, $b\le a < 2\frac{n+1}{n}b$ and  $c=\lfloor c(n,a,b)\rfloor +1$. There exists a $c$-separating $a$-list assignment $L$ of $P_{n+1}$ with $|L_1|=|L_{n+1}|=b$
such that $P_{n+1}$ is not $(L,b)$-colorable. Moreover, the same holds if in addition, the list assignment $L$ is such that $L_1=L_{n+1}$ or $L_1\cap L_{n+1}=\emptyset$.
\end{lemma}

\begin{proof}
We provide a counter-example in each of the two following cases.
\begin{description}
\item [Case 1.] $b\le a < \frac{2n-1}{n-1}b$. 

We show that for $c=\lfloor\frac{n-1}{n}(a-b)\rfloor +1$ the following list assignment $L$ is a $c$-separating $a$-list assignment of $P_{n+1}$ is such that $L_1=L_{n+1}$, $|L_1|=b$, but no $(L,b)$-coloring exists. In order to have a compact representation and to shorten the proof, $L$ is described in a graphical way showing the intersections of sublists composing the lists $L_i$ (each box represent a color subset and the number inside a box indicates its size):  

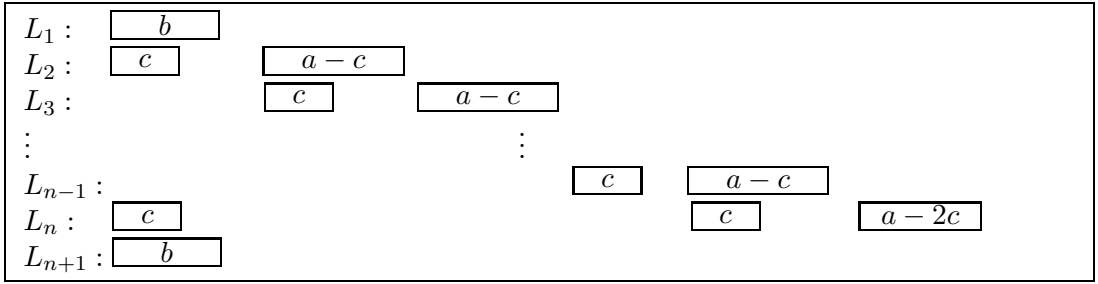
\begin{center}
 \fbox{ \begin{minipage}{.95\textwidth}
$L_1:\ \ \ $ \framebox(40,10){$b$}
 
$L_2:\ \ \ $ \framebox(25,10){$c$}\hspace*{1cm} \framebox(52,10){$a-c$}

$L_3:\ \ \ $ \hspace*{1.9cm} \framebox(25,10){$c$}\hspace*{1cm} \framebox(52,10){$a-c$}

$\vdots$\hspace*{6.4cm}$\vdots$

$L_{n-1}:\ $ \hspace*{5.8cm} \framebox(25,10){$c$}\hspace*{.5cm} \framebox(52,10){$a-c$}

$L_n:\ \ \ $ \framebox(25,10){$c$}\hspace*{6.6cm} \framebox(25,10){$c$}\hspace*{1.2cm} \framebox(45,10){$a-2c$}

$L_{n+1}:$ \framebox(40,10){$b$}

\end{minipage}}
\end{center}

Since $a< \frac{2n-1}{n-1}b$, then $c=  \lfloor\frac{n-1}{n}(a-b)\rfloor +1 < \lfloor\frac{n-1}{n}\frac{2n-1-n+1}{n-1}b\rfloor +1= b+1$, i.e. $c\le b$. Moreover, since $\frac{n-1}{n} < 1$, we have $c< a-b +1$, hence $a> b+c-1 \ge 2c-1$. Consequently, the list assignment $L$ is well defined.

Observe that each color of each list $L_i$, $2\le i\le n-1$, can be used on only one vertex. Also, for any color $k$ shared by the lists of $L_1$ and $L_{n+1}$, we have $\alpha(P_{n+1},L,k) = 2$. 
Hence, we have 
$$\Sigma(L)= 2b+ (n-2)(a-c)+a-2c = (n-1)(a -c) +2b-c.$$
In order the amplitude condition to be satisfied, we must have $\Sigma(L)\ge (n+1)b$, i.e., $(n-1)(a -c) +2b-c \ge (n+1)b$, which is equivalent to $c \le \frac{n-1}{n}(a-b)$. Consequently, $P_{n+1}$ is not $(L,b)$-colorable when $c> \frac{n-1}{n}(a-b)$.

\item[Case 2.] $\frac{2n-1}{n-1}b \le a < 2\frac{n+1}{n}b$.

We show that for $c=\lfloor\frac{n-1}{n-2}(a-b) - \frac{2}{n-2}b\rfloor+1$ there is a $c$-separating $a$-list assignment $L$ of $P_{n+1}$ for which $L_1=L_{n+1}$, $|L_1|=b$, but no $(L,b)$-coloring exists.
First, we show that $a\ge 2c-1$: Since $c\le \frac{n-1}{n-2}a-\frac{n+1}{n-2}b +1$ and $a< 2\frac{n+1}{n}b$, i.e. $b > \frac{n}{2n+2}a$, then $2c-2 \le \frac{2n-2}{n-2}a-\frac{2n+2}{n-2}b < \frac{2n-2}{n-2}a-\frac{2n+2}{n-2}\frac{n}{2n+2}a = a$.\\
Second, we show that $c\ge b+1$: As $c> \frac{n-1}{n-2}a-\frac{n+1}{n-2}b$ and since $a\ge \frac{2n-1}{n-1}b$, then we obtain $c> \frac{n-1}{n-2}\frac{2n-1}{n-1}b -\frac{n+1}{n-2}b = b$.

Depending on the value of $a$ and $c$, we consider two subcases and for each we provide a $c$-separating $a$-list assignment $L$ for which $|L_1|=b$ and no $(L,b)$-coloring exists.

\item[Subcase 2.a.] $a\ge 2c$.

Consider the list assignment $L$ defined as follows:
\begin{center}
 \fbox{ \begin{minipage}{.96\textwidth}
$L_1:\ \ \ $ \framebox(40,10){$b$}
 
$L_2:\ \ \ $ \framebox(40,10){$b$}\hspace*{.5cm} \framebox(40,10){$a-b$}

$L_3:\ \ \ $ \hspace*{1.9cm} \framebox(25,10){$c$}\hspace*{.8cm} \framebox(52,10){$a-c$}

$\vdots$\hspace*{6.4cm}$\vdots$

$L_{n-1}:\ $ \hspace*{5.5cm} \framebox(25,10){$c$}\hspace*{.7cm} \framebox(52,10){$a-c$}

$L_n:\ \ \ $ \framebox(40,10){$b$}\hspace*{5.95cm} \framebox(25,10){$c$}\hspace*{1.2cm} \framebox(48,10){$a-c-b$}

$L_{n+1}:$ \framebox(40,10){$b$}
\end{minipage}}
\end{center}

Since we are in the case that $a\ge 2c$ and $c\ge b+1$, the list is well defined.
We have 
$$\Sigma(L)= 2b+ a-b + (n-3)(a-c)+a-c-b= (n-1)a - (n-2)c.$$
Therefore, $\Sigma(L)\ge (n+1)b$ implies $c \le \frac{n-1}{n-2}(a-b)-\frac{2}{n-2}b$. Consequently, $P_{n+1}$ is not $(L,b)$-colorable when $c> \frac{n-1}{n-2}(a-b)-\frac{2}{n-2}b$.

\item[Subcase 2.b.] $a=2c-1$.

We are in the case that $a< 2\frac{n+1}{n}b$, i.e., $a < 2b + \frac{2b}{n}$. Hence, as $a=2c-1$, $c$ satisfies
\begin{equation}\label{eq:1}
 c < b + \frac{b}{n} + \frac12
\end{equation}

If $n$ is odd, consider the list assignment $L$ defined as follows:
\begin{center}
 \fbox{ \begin{minipage}{.97\textwidth}
$L_1:\ \ \ $ \framebox(40,10){$b$}
 
$L_2:\ \ \ $ \framebox(40,10){$b$}\hspace*{.2cm} \framebox(52,10){$2c-b-1$}

$L_3:\ \ \ $ \hspace*{1.6cm} \framebox(36,10){$c$}\hspace*{.6cm} \framebox(32,10){$c-1$}

$L_4:\ \ \ $ \hspace*{3.6cm} \framebox(32,10){$c-1$}\hspace*{.2cm} \framebox(36,10){$c$}

$\vdots$\hspace*{6.4cm}$\vdots$

$L_{n-1}:\ $ \hspace*{6.8cm} \framebox(32,10){$c-1$}\hspace*{.2cm} \framebox(36,10){$c$}

$L_n:\ \ \ $ \framebox(40,10){$b$}\hspace*{7cm} \framebox(36,10){$c$}\hspace*{.2cm} \framebox(48,10){$c-b-1$}

$L_{n+1}:$ \framebox(40,10){$b$}
\end{minipage}}
\end{center}

We have 
$$\Sigma(L)=2b+ 2c-b-1 + \frac{n-3}{2}(c-1) + \frac{n-3}{2}c + c-b-1 = nc - \frac{n+1}{2}.$$
Therefore the amplitude condition is not satisfied if $nc - \frac{n+1}{2} < (n+1)b$, i.e., if $c< b + \frac{b}{n} + \frac12 + \frac{1}{2n}$ which is true by Equation~\ref{eq:1}.

If $n$ is even, consider the list assignment $L$ defined as follows:
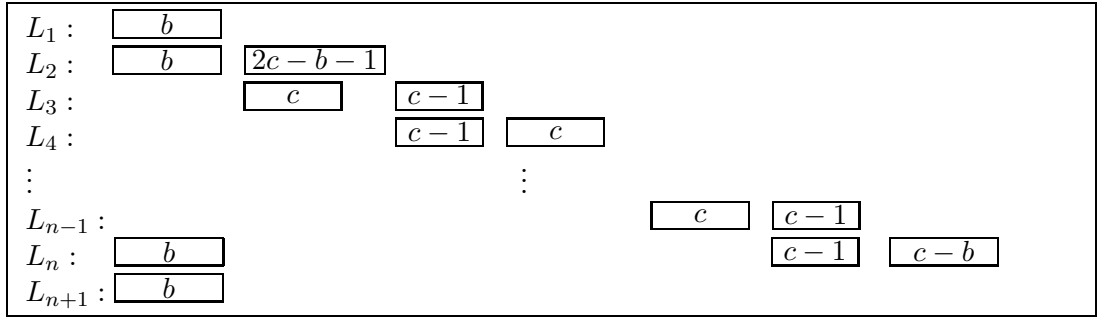
\begin{center}
 \fbox{ \begin{minipage}{.95\textwidth}
$L_1:\ \ \ $ \framebox(40,10){$b$}
 
$L_2:\ \ \ $ \framebox(40,10){$b$}\hspace*{.2cm} \framebox(52,10){$2c-b-1$}

$L_3:\ \ \ $ \hspace*{1.6cm} \framebox(36,10){$c$}\hspace*{.6cm} \framebox(32,10){$c-1$}

$L_4:\ \ \ $ \hspace*{3.6cm} \framebox(32,10){$c-1$}\hspace*{.2cm} \framebox(36,10){$c$}

$\vdots$\hspace*{6.4cm}$\vdots$

$L_{n-1}:\ $ \hspace*{6.8cm} \framebox(36,10){$c$}\hspace*{.2cm} \framebox(32,10){$c-1$}

$L_n:\ \ \ $ \framebox(40,10){$b$}\hspace*{7.1cm} \framebox(32,10){$c-1$}\hspace*{.3cm} \framebox(40,10){$c-b$}

$L_{n+1}:$ \framebox(40,10){$b$}
\end{minipage}}
\end{center}

We have 
$$\Sigma(L)=2b+ 2c-b-1 + \frac{n-2}{2}(c-1) + \frac{n-4}{2}c + c-b= nc - \frac{n}{2}.$$
Therefore the amplitude condition is not satisfied if $nc - \frac{n}{2} < (n+1)b$, i.e., if $c< b + \frac{b}{n} + \frac12$ which is true by Equation~\ref{eq:1}.
\end{description}

The counter-examples presented are such that $L_1=L_{n+1}$, but they can be easily modified in order that $L_1\cap L_{n+1}=\emptyset$ without changing the conclusion. For this, in each of the above four list assignments, instead of using the $b$ colors of $L_1$ for $L_n$ and $L_{n+1}$, just take $b$ new colors (not used by any list $L_i$, $1\le i\le n-1$).
\end{proof}

\begin{theorem}\label{thm:fCn}
For any $a\ge b\ge 1$ and any $n\ge 4$, 
 \[\fsep(C_n,a,b) = \left\{\begin{array}{lccl}
\left\lfloor\frac{n-1}{n}(a-b)\right\rfloor, & b&\le &a< \frac{2n-1}{n-1}b\\
\left\lfloor\frac{n-1}{n-2}(a-b) - \frac{2}{n-2}b\right\rfloor, & \frac{2n-1}{n-1}b &\le &a < 2\frac{n+1}{n}b\\
a,& 2\frac{n+1}{n}b&\le &a.
                            \end{array}
\right.\]
\end{theorem}

\begin{proof}
First, if $2\frac{n+1}{n}b\le a$, then we know from~\cite{AGT16} that $C_n$ is $(a,b,a)$-free-choosable.

For the two other cases, given a $c$-separating $a$-list assignment $L$ of $C_n$ with $|L_1|=b$ and $c\le \lfloor c(n,a,b)\rfloor$, we consider the list assignment  $L'$ on $P_{n+1}$ obtained from $L$ on $C_n$ by cutting the cycle at $x_1$ as in Remark~\ref{rem}.
By Lemma~\ref{lem:Pn+1}, $P_{n+1}$ is $(L',b)$-colorable. Hence $C_n$ is $(L,b)$-colorable for $c\le \lfloor c(n,a,b)\rfloor$.
For the converse, if $c=\lfloor c(n,a,b)\rfloor +1$, then we know, by Lemma~\ref{lem:Pn+1onlyif} that there exists a $c$-separating $a$-list assignment $L'$ of $P_{n+1}$ with $L'_1=L'_{n+1}$ and $|L'_1|=b$, such that $P_{n+1}$ is not $(L',b)$-colorable. Therefore, identifying the vertices $x_1$ and $x_{n+1}$ of $P_{n+1}$, we obtain a $c$-separating $a$-list assignment $L$ on the cycle $C_n$ such that $|L_1|=b$ and no $(L,b)$-coloring of $C_n$ exists. 
\end{proof}

%
%

It only remains to determine the free separation number of the cycle of length $3$ which has a special behavior.
\begin{proposition}\label{prop:fC3}
 \[\fsep(C_3,a,b) = \left\{\begin{array}{ll}
\lfloor\frac{2}{3}(a-b)\rfloor, & b\le a< \frac{7}{4}b\\
2a-3b, & \frac{7}{4}b \le a < 3b\\
a,& 3b\le a.
                            \end{array}
\right.\]
\end{proposition}

\begin{proof}We consider the three following cases depending on $a$: 
\begin{description}
\item[Case 1.] $b\le a < \frac{7}{4}b$.

Let $c=\lfloor\frac{2}{3}(a-b)\rfloor+1$. We prove that $C_3$ is not $(a,b,c)$-free-choosable. For this, we give a $c$-separating $a$-list assignment $L$ for which $|L_1|=b$ and no $(L,b)$-coloring exists.  
%
%
%
%
\begin{center}
 \fbox{ \begin{minipage}{0.52\textwidth}
$L_1:$ \framebox(60,10){$b$}
 
$L_2:$ \framebox(25,10){$c$}\hspace*{1.5cm} \framebox(40,10){$a-c$}

$L_3:$ \hspace*{.9cm}\framebox(35,10){$b-c$}\hspace*{.3cm} \framebox(25,10){$c$}\hspace*{.8cm} \framebox(32,10){$a-b$}
\end{minipage}}
\end{center}

In order this list to be well defined, we must have $c\le b$. If $b=1$ then $a=1$ and thus $c=1=b$. If $b\ge 2$ then $c\le \frac23 (a-b) +1\le b$ if $a\le \frac{5b-3}{2}$ which is true since $\frac74 b\le \frac{5b-3}{2}$ for $b\ge 2$.
We must also have $a-c\ge c$ which is true since $a-2c \ge a -4(a-b)/3 = 4b/3 -a/3 \ge 0$ as $a< \frac{7}{4}b$. Moreover, we have $\Sigma(L)= b+ a-c+a-b = 2a-c$.  Since $c > \frac{2}{3}(a-b)$, $\Sigma(L)<3b$ if $6a -2a +2b< 9b$, i.e., if $a<\frac74 b$ which is true by hypothesis. Hence the amplitude condition is not satisfied and thus $C_3$ is not $(L,b)$-colorable.

Now, we prove that  $C_3$ is $(a,b,c')$-free-choosable with $c'=c-1= \lfloor\frac{2}{3}(a-b)\rfloor$. First, observe that we have $b- \frac{2}{3}(a-b) = \frac{5b-2a}{3}  \ge 0$, hence $b\ge c'$.
Since $C_3$ is a complete graph, by Cropper et al.'s result~\cite{CGHHJ}, the amplitude condition is sufficient in order it is $(L,b)$-colorable. This is clearly true for a subgraph of $C_3$ reduced to one vertex and for a subgraph of two vertices since for any $c'$-separating $a$-list assignment $L$ with $|L_1|=b$, we have $|L_2\cup L_3|\ge 2a-c'\ge 2b$ as $a\ge b$ and, for $i=2$ or $3$, $|L_1\cup L_i|\ge b+ a-\min(b,c') = b+a-c'\ge 2b$ since $c'<a-b$. 
For the whole graph, since $c'\le \frac{2}{3}(a-b)$, we have:
$$\Sigma(L)\ge |L_1| + |L_2\setminus L_1| + |L_3\setminus (L_1\cup L_2)| \ge b + a-c'+a-2c' = 2a+b-3c' \ge 3b.$$ 

\item[Case 2.] $\frac{7}{4}b \le a < 3b$.

Let $c=2a-3b+1$. We prove that $C_3$ is not $(a,b,c)$-free-choosable. For this, we give a $c$-separating $a$-list assignment $L$ for which $|L_1|=b$ and no $(L,b)$-coloring exists:
 We present a list for each of the two following subcases  depending on whether $a\ge 2b$ or not.
 
 If $a\ge 2b$, $L$ is made up as follows:
\begin{center}
 \fbox{ \begin{minipage}{0.56\textwidth}
$L_1:$ \framebox(50,10){$b$}
 
$L_2:$ \framebox(50,10){$b$}\hspace*{.5cm} \framebox(45,10){$a-b$}

$L_3:$ \framebox(50,10){$b$}\hspace*{.5cm} \framebox(35,10){$c-b$}\hspace*{.8cm} \framebox(50,10){$a-c$}
\end{minipage}}
\end{center}
For this list to be well defined, we must have $a-c\ge 0$ which is true by hypothesis. In order $L$ to be $c$-separating, we must have $b\le c=2a-3b+1$, i.e., $a\ge 2b-\frac12$ which is true.

Moreover, we have $\Sigma(L)= b+ a-b+a-c = 2a-c = 3b-1 < 3b$, hence the amplitude condition is not satisfied and thus $C_3$ is not $(L,b)$-colorable.

If $a< 2b$, $L$ is made up as follows:
\begin{center}
 \fbox{ \begin{minipage}{0.57\textwidth}
$L_1:$ \framebox(60,10){$b$}
 
$L_2:$ \framebox(32,10){$c$}\hspace*{1.6cm} \framebox(45,10){$a-c$}

$L_3:$ \hspace*{1.05cm} \framebox(28,10){$b-c$}\hspace*{.6cm} \framebox(35,10){$c$}\hspace*{.8cm} \framebox(40,10){$a-b$}
\end{minipage}}
\end{center}
For this list to be well defined, we must have $c=2a-3b+1 \ge b-c= 4b-2a-1$, i.e., $a\ge \frac{7}{4}b -1/2$, which is true by hypothesis. We also have $c\le a-c$ since $a< 2b$.

We have $\Sigma(L)= b+a-c+a-b=2a-c= 3b-1 < 3b$, hence the amplitude condition is not satisfied and thus $C_3$ is not $(L,b)$-colorable.

Now, in both the cases $a\ge 2b$ and $a<2b$, we prove that  $C_3$ is $(a,b,c')$-free-choosable with $c'=c-1= 2a-3b$. Again, it is sufficient to verify that the amplitude condition is satisfied by any $c'$-separating $a$-list assignment $L$ with $|L(x_1)|=b$. This is clearly true for a subgraph of $C_3$ reduced to one vertex and for a subgraph of two vertices since we have $|L_2\cup L_3|\ge 2a-c'=3b\ge 2b$ and $|L_1\cup L_i|\ge b+ a-\min(b,c')\ge 2b$ for $i=2$ or $3$. For the whole graph, we have 
$$\Sigma(L)\ge |L1| + |L2\setminus L_1| + |L_3\setminus (L_1\cup L_2)| \ge b + a-\alpha+a-\beta,$$
with $\alpha=|L_1\cap L_2|$ and $\beta = |(L_1\cup L_2)\cap L_3|$. Let $\beta=\beta_1 + \beta_2-\gamma$, where $\beta_1=|L_1\cap L3|$, $\beta_2=|L_2\cap L_3|$ and $\gamma = |L_3\cap L_2\cap L_1|$. Then we have $\Sigma(L)\ge b +2a -\alpha -\beta_1-\beta_2 + \gamma$. As, by definition, $\alpha + \beta_1 + \gamma \le b$ and $\beta_2\le c'$, then we obtain $\Sigma(L)\ge b+2a-b-c' = 2a-c' = 3b$.

 \item[Case 3.] $a\ge 3b$.
 In this case, $C_3$ is trivially $(a,b,a)$-free-choosable and the result follows.
\end{description}
In conclusion, in each of the three cases, we have shown that the maximum value of $c$ for which $C_3$ is $(a,b,c)$-free choosable is the one given in the statement.
\end{proof}

%

\section{Outerplanar graphs}
\label{sec:fsepCact}

An {\em outerplanar graph} is a graph that has a planar drawing in which all vertices belong to the outer face of the drawing. 
For an outerplanar graph $G$, we denote by $\mathcal{T}_G$ the {\em weak dual} of $G$, i.e., the graph whose vertex set is the set of all inner faces of $G$, and $E(\mathcal{T}_G)=\{\alpha \beta |\ \alpha \text{ and  } \beta \text{ share a common edge} \}$.
 A {\em cactus} is a graph in which every edge is part of at most one cycle. Cactuses form a subclass of outerplanar graphs. The {\em girth} of a graph $G$ is the length of a shortest cycle in $G$.

We will use the fact that if an outerplanar graph $G$ is 2-connected, then $\mathcal{T}_G$ is a tree.
Moreover, in order to restrict our argument to 2-connected graphs, we first show the following which has been proved in~\cite{AGT10} in the case $c=a$: 
\begin{lemma}\label{le2connex}
Let $a,b,c$ be integers and let $G_1 ,G_2$ be two $(a,b,c)$-free-choosable graphs. Then the graph obtained from $G_1$ and $G_2$ by identifying any vertex of $G_1$ with any vertex of $G_2$ is $(a,b,c)$-free-choosable.
\end{lemma}
\begin{proof}Let $G$ be the graph obtained by identifying vertex $x_1$ of $G_1$ with vertex $x_2$ of $G_2$, resulting in a vertex named $x$.
Let $y \in V(G)$ and let $L$ be a list assignment of $G$ with $|L(v)|=a$ for $v\in V(G)\setminus\{y\} $ and $|L(y)|=b$ (i.e. $y$ is the precolored vertex).
Assume without loss of generality, that $y\in V(G_1)$. Let $L_i, i=1,2$, be the sublist assignment of $L$ restricted to vertices of $G_i$.
As $G_1 ,G_2$ are both $(a,b,c)$-free-choosable, there exists an $(L_1,b)$-coloring $c_1$ of $G_1$ and an  $(L_2,b)$-coloring $c_2$ of $G_2$ such that $c_2(x)=c_1(x)$ (i.e. $x$ is the precolored vertex of $G_2$). The union of colorings $c_1$ and $c_2$ is an $(L,b)$-coloring of $G$.
\end{proof}

\begin{lemma}\label{le:fCn}
 For any positive integers $a,b,n$ and $i$, if $n\ge 4$ then $$\fsep(C_n,a,b)\le \fsep(C_{n+i},a,b).$$ 
 
 Moreover, we have $$\min(\fsep(C_3,a,b),\fsep(C_{n},a,b))=
 \left\{\begin{array}{ll}
\lfloor\frac{2}{3}(a-b)\rfloor     & \mbox{ if }  b\le a < \frac{7}{4}b\\
2a-3b                         & \mbox{ if } \frac{7}{4}b\le a \le \frac{2n+1}{n+1}b\\
& \mbox{ or } \frac{2n+2}{n}b\le a < 3b\\
\lfloor\frac{n-1}{n}(a-b)\rfloor & \mbox{ if }  \frac{2n+1}{n+1}b < a < \frac{2n-1}{n-1}b\\      
\lfloor\frac{n-1}{n-2}(a-b)-\frac{2}{n-2}b\rfloor & \mbox{ if } \frac{2n-1}{n-1}b\le a < \frac{2n+2}{n}b\\
a                                  & \mbox{ if } 3b\le a.\\
      \end{array}\right.$$
\end{lemma}
\begin{proof}
The values of $\fsep$ for a cycle are given in Theorem~\ref{thm:fCn}. For the first assertion, we show that in each of the following cases we have $A=c(n,a,b)\le B=c(n+1,a,b)$ (recall that $\fsep(C_n,a,b)=\lfloor c(n,a,b)\rfloor$).
\begin{itemize}
 \item $b\le a < \frac{2(n+1)-1}{n}b$.
 
 Since $\frac{n-1}{n}\le \frac{n}{n+1}$, we have $A=\frac{n-1}{n}(a-b)\le B=\frac{n}{n+1}(a-b)$.
 
 \item $\frac{2n+1}{n}b \le a < \frac{2n-1}{n-1}b$.
 
 We want $A=\frac{n-1}{n}(a-b)\le B=\frac{n(a-b) -2b}{n-1}$, i.e. $(n-1)^2(a-b)\le n^2(a-b)-2bn$. This can be rewritten as $a\geq \frac{4n-1}{2n-1}b$, which is true since $a\ge \frac{2n+1}{n}b$.
 
 \item $\frac{2(n+1)-1}{n}b\le a < 2\frac{n+2}{n+1}b$.
 
 We want $A=\frac{(n-1)(a-b) -2b}{n-2} \le B=\frac{n(a-b) -2b}{n-1}$, which simplifies to $a\le 3b$ an is true by hypothesis.
 
 \item $2\frac{n+2}{n+1}b\le a \le 2\frac{n+1}{n}b$.
 
 We want $A=\frac{(n-1)(a-b) -2b}{n-2} \le B=a$, i.e., $a\le (n+1)b$ which is true.
 
\end{itemize}

 For the second assertion, by Proposition~\ref{prop:fC3} and Theorem~\ref{thm:fCn}, we infer the desired inequalities between $A=\fsep(C_3,a,b)$ and $B=\fsep(C_n,a,b)$, $n\ge 4$.
 \begin{itemize}
  \item For $b\le a < \frac{7}{4}b$, we have $A=\lfloor\frac{2}{3}(a-b)\rfloor \le B=\lfloor\frac{n-1}{n}(a-b)\rfloor$.
  \item For $\frac{7}{4}b\le a \le \frac{2n+1}{n+1}b$, we have $A=2a-3b \le B=\lfloor\frac{n-1}{n}(a-b)\rfloor$.
  \item For $\frac{2n+1}{n+1}b < a < \frac{2n-1}{n-1}b$, we have $A=2a-3b \ge B=\lfloor\frac{n-1}{n}(a-b)\rfloor$.
  \item For $\frac{2n-1}{n-1}b\le a < \frac{2n+2}{n}b$, we have $A=2a-3b \ge B=\lfloor\frac{n-1}{n-2}(a-b)-\frac{2}{n-2}b\rfloor$.
  \item For $\frac{2n+2}{n}b\le a < 3b$, we have $A=2a-3b \le B=a$.
  \item For $3b\le a$, we have $A= B=a$.
  
 \end{itemize}

\end{proof}

From the free-separation number of the cycle we derive the free-separation number of cactuses. 
\begin{theorem}
 \label{th:fcact}
 Let $G$ be a cactus with finite girth $g$ and let $a\ge b\ge 1$ be integers. Then If $g\ge 4$ or if $G$ has only cycles of length three, then 
 \[\fsep(G,a,b) = \fsep(C_g,a,b).\]
 
 Otherwise, if $G$ contains at least one triangle and if $\ell$ is the length of a shortest cycle of $G$ greater than three, then
  \[\fsep(G,a,b) = \left\{\begin{array}{ll}
                           \fsep(C_{\ell},a,b) &\mbox{ if } \frac{2\ell+1}{\ell+1}b < a < \frac{2\ell+2}{\ell}b,\\
                           \fsep(C_3,a,b) & \mbox{ otherwise.}
                          \end{array}\right.
\]
 
\end{theorem}

\begin{proof}
Let $G$ be a cactus of finite girth $g$ and let $a,b,c$ be integers. Then, each of its blocks $B_1, B_2, \ldots, B_r$ is either a cycle (of length at least $g$) or a single edge, and they are connected in a treelike structure.  

We first show that if $g\ge 4$ or if all cycles of $G$ are of length 3, then $G$ is $(a,b,c)$-free-choosable if and only if $c\le \fsep(C_g,a,b)$.
Let $c\le fsep(C_g,a,b)$ and $L$ be a $c$-separating $a$-list assignment of $G$ such that $|L(x_1)|=b$ (i.e., $x_1$ is the precolored vertex) and suppose without loss of generality that $x_1\in B_1$. 
By Lemma~\ref{le2connex}, it is sufficient to prove that each block is $(a,b,c)$-choosable. By Lemma~\ref{le:fCn}, if $g\ge 4$, then each block consisting of a cycle $C$ of length at least $g$ can be colored if $c\le \fsep (C_g,a,b)$. Also, trivially, any edge is $(a,b,c)$-choosable if $c\le a-b$, hence $\fsep(K_2,a,b)\ge a-b\ge  \fsep(C_g,a,b)$. Therefore, there exists an $(L,b)$-coloring of $G$ and thus it is $(a,b,c)$-free-choosable. Moreover, since $G$ contains a cycle of length $g$, then  $\fsep(G,a,b) \le \fsep(C_g,a,b)$.

Second, if $G$ contains a triangle and cycles of length greater than three, let $\ell$ be the length of the shortest cycle of length at least 4. Then, by Lemma~\ref{le:fCn}, $\fsep(C_3,a,b)> \fsep(C_{\ell},a,b)$ if and only if $\frac{2\ell+1}{\ell+1}b < a < \frac{2\ell+2}{\ell}b$. We then proceed as for the first part of the proof, but with $c\le \min(\fsep(C_g,a,b),\fsep(C_{\ell},a,b))$.
\end{proof}

As the free-separation number is a lower bound of the separation number, Theorem~\ref{th:fcact} provides a lower bound on the separation number of cactuses. Moreover, the following result shows that this lower bound is tight in some sense.

\begin{proposition}\label{prop:sepCact}
 For any $a\ge b\ge 1$, and any $p\ge 3$, there exists a cactus $G$ of girth $p$ such that $\sep(G,a,b)=\fsep(G,a,b)$.
\end{proposition}

\begin{proof}
Let $c=\fsep(C_p,a,b)$, and let $k= {b\choose a}$.
 Let $G$ be the graph obtained by joining $k$ copies  $C^1, C^2, \ldots, C^k$ of the cycle $C_p$ of length $p$ at a shared universal vertex $x_1$ (see Figure~\ref{fig:2C4} for an illustration in the case $a=2$ and $b=1$). 
Let $B_i$, $1\le i\le k$ be the sets of $b$-subsets of $\{1,\ldots, a\}$. 
In order to show that $\sep(G,a,b)=c$, we construct a $(c+1)$-separating $a$-list assignment $L$ of $G$ for which no $(L,b)$-coloring exists as follows:
For each $i, 1\le i\le k$, let $L^i$ be a $(c+1)$-separating $a$-list assignment of $C^i$ for which $L(x_1)=B_i$ and other lists of colors only use colors from $B_i$ and from a set of colors $A$ with $A\cap \{1,\ldots, a\}=\emptyset$ and such that $C^i$ is not $(L^i,b)$-colorable. Since $\fsep(C_p,a,b)=c$, such an assignment exists.
Let now $L$ be the list assignment of $G$ defined by
\[L(y)=\left\{\begin{array}{ll}
               \{1,\ldots, a\}, & \mbox{ if } y=x_1\\
               L^i(y) , & \mbox{ if } y\ne x_1 \mbox{ and } y\in C^i.
              \end{array}\right. 
\]
Then, by construction, $L$ is a $(c+1)$-separating $a$-list assignment of $G$.  Moreover, whatever the choice of the set of $b$ colors for the vertex $x_1$, there will be a cycle $C^i$ on which the $(L,b)$-coloring cannot be completed. Therefore, $\sep(G,a,b)\le c$ and since $\sep(G,a,b)\ge \fsep(G,a,b)$, we have that $\sep(G,a,b) = \fsep(G,a,b)$.
\end{proof}
 
 

\begin{theorem}
 \label{th:outer}
Let $G$ be an outerplanar graph with finite girth $g\ge 5$ and let $a,b$ be integers, $a\ge b\ge 1$. Then we have 
\[ \fsep(C_{g-1},a,b)\le \fsep(G,a,b)\le \fsep(C_{g},a,b). \]
\end{theorem}

\begin{proof}
 First, we are going to prove that $G$ is $(a,b,c)$-free choosable for $c=\fsep(C_{g-1},a,b)$. By Lemma~\ref{le2connex}, we may suppose that $G$ is 2-connected.  Let $\alpha_1, \alpha_2,\ldots, \alpha_k$ be the inner faces of $G$ and let $r$ be any vertex of any face, say $\alpha_1$.
 Let $L$ be a $c$-separating $a$-list assignment of $G$ such that $|L(r)|=b$. We define an $(L,b)$-coloring of $G$ by coloring the vertices of the faces, following a BFS order on the tree $\mathcal{T}_G$, starting with the face $\alpha_1$. Since $g\ge 5$, by Lemma~\ref{le:fCn}, we have $c=\fsep(C_{g-1},a,b) \le \fsep(\alpha_1,a,b)$, thus there exists an $(L,b)$-coloring of $\alpha_1$. At each step, when coloring the vertices of a face $\alpha_i=(x_1,x_2,\ldots, x_{\ell})$, this face shares an edge with at most one face $\alpha_j$ with already colored vertices. Assume without loss of generalities, that $x_1x_{\ell}\in \alpha_i \cap \alpha_j$. Then we have a path $P=(x_1,x_2,\ldots, x_{\ell})$ of length $\ell$ with precolored endvertices and since $\ell\ge g-1$, we have $c\ge \fsep(C_{\ell,a,b})=\lfloor c(\ell,a,b)\rfloor$. Therefore, by Lemma~\ref{lem:Pn+1}, there exists an $(L,b)$-coloring of $P$. By iterating the process on each face, we obtain an $(L,b)$-coloring of the whole graph $G$, hence proving that $\fsep(G,a,b)\ge c$.

Second, since $G$ contains a cycle of length $g$, then $\fsep(G,a,b)\le \fsep(C_{g},a,b)$.
\end{proof}

Remark that lower bounds for the free-separation number of an outerplanar graph with girth four can be derived using a proof similar with the one  for $g\ge 5$ but the formula will be more complex. For outerplanar graphs of girth $3$, we cannot use Lemma~\ref{lem:Pn+1} anymore since it needs the path $P$ being of length at least $3$.

\section{Concluding remarks}

We have determined the separation and free-separation number of the cycle and the free separation number of cactuses, and only bounds for general outerplanar graphs $G$ of girth $g\ge 5$. For some values of $g,a,b$ the lower and upper bounds of Theorem~\ref{th:outer} are equal, but for some not.
For instance, we have $\fsep(C_4,9,4)=3$ and $\fsep(C_5,9,4)=4$.
We conjecture that for any $a\ge b\ge 1$ and any $g\ge 5$, there exists an outerplanar graph $G$ of girth $g$ such that $\sep(G,a,b)=\fsep(G,a,b)=\fsep(C_{g-1},a,b)$.

The problem seems also hard for other simple graphs such as the complete graph $K_n$.
In~\cite{KTV98}, the assymptotic on the minimum $a$ such that $K_n$ is $(a,1,1)$-choosable is given.
 For $n=3$, $\sep(K_3,a,b)$ and $\fsep(K_3,a,b)$ are given in Theorem~\ref{th:Cn} and Proposition~\ref{prop:fC3}, respectively. For $n=4$ we are able to determine both numbers for any values of $a$ and $b$, but there are many cases in the formulae. For $n=5$ even many more cases have to be considered. We conjecture that $f_n(a/b) =\fsep(K_n,a,b)$ is a piecewise linear function with a number of pieces growing exponentially with $n$.

\end{document}